\documentclass[a4paper,12pt]{article}
\usepackage{amsmath,amssymb}
\usepackage{bm}

\bmdefine{\xxx}{x}
\bmdefine{\aaa}{a}
\bmdefine{\bbb}{b}
\bmdefine{\eee}{e}
\bmdefine{\vvv}{v}
\bmdefine{\www}{w}
\bmdefine{\mmm}{m}
\bmdefine{\zerovec}{0}

\DeclareMathAlphabet{\mathscr}{U}{rsfs}{m}{n}

\newcommand{\CCC}{\mathbb{C}}
\newcommand{\PPP}{\mathbb{P}}
\newcommand{\RRR}{\mathbb{R}}

\newcommand{\NNN}{\mathbb{N}}

\newcommand{\ZZZ}{\mathbb{Z}}

\newcommand{\msOOO}{\mathscr{O}}
\newcommand{\msVVV}{\mathscr{V}}
\newcommand{\msUUU}{\mathscr{U}}
\newcommand{\msIII}{\mathscr{I}}

\newcommand{\UUUUU}{{\mathcal U}}

\newcommand{\height}{\mathrm{ht}}

\newcommand{\define}{\mathrel{:=}}

\newcommand{\rank}{\mathrm{rank}}

\newcommand{\glin}{{\mathrm{GL}}}

\newcommand{\transpose}{^\top}

\newcommand{\image}{{\mathrm{Im}}}
\newcommand\diag{\mathrm{Diag}}

\newcommand\nonsing{nonsingular}

\newcommand{\grank}{{\mathrm{grank}}}
\newcommand{\trank}{{\mathrm{trank}}}
\newcommand\fl{\mathrm{fl}}
\newcommand\kprod{\otimes_\mathrm{Kr}}

\newtheorem{thm}{Theorem}[section]
\newtheorem{fact}[thm]{Fact}
\newtheorem{example}[thm]{Example}
\newtheorem{lemma}[thm]{Lemma}
\newtheorem{cor}[thm]{Corollary}
\newtheorem{definition}[thm]{Definition}

\newtheorem{remark}[thm]{Remark}

\makeatletter
\newcommand{\mysloppy}{\tolerance 9999 \hfuzz .5\p@ \vfuzz .5\p@}
\makeatother
\title{%
Typical ranks of semi-tall real 3-tensors}
\author{%
Toshio SUMI\footnote{Faculty of Arts and Science, Kyushu University, Fukuoka, Japan},
Mitsuhiro MIYAZAKI\footnote{Department of  Mathematics, Kyoto University of Education, Kyoto, Japan}
and
Toshio SAKATA\footnote{Emeritus professor, Kyushu University, Fukuoka, Japan}}
\date{Version of \today}
\date{}

\begin{document}
\sloppy

\maketitle

\begin{abstract}
Let $m$, $n$ and $p$ be integers with $3\leq m\leq n$ and
$(m-1)(n-1)+1\leq p\leq (m-1)m$.
We showed in previous papers that if $p\geq (m-1)(n-1)+2$, then
typical ranks of $p\times n\times m$-tensors over the real number
field are $p$ and $p+1$ if and only if there exists a \nonsing\ 
bilinear map $\RRR^m\times \RRR^n\to\RRR^{mn-p}$.
We also showed that the ``if'' part also valid in the case where
$p=(m-1)(n-1)+1$.
In this paper, we consider the case where $p=(m-1)(n-1)+1$
and show that the typical ranks of $p\times n\times m$-tensors
over the real number field are $p$ and $p+1$ in several
cases including the case where there is no \nonsing\  bilinear
map $\RRR^m\times \RRR^n\to\RRR^{mn-p}$.
In particular, we show that the ``only if'' part of the above mentioned
fact does not valid for the case $p=(m-1)(n-1)+1$.
\end{abstract}

\section{Introduction}

Tensor rank is a subject which is widely 
studied in both pure and applied mathematics.
A high dimensional array of datum is called a tensor in the field of
data analysis.
Precisely, let $N_1$, \ldots, $N_d$ be positive integers.
A $d$-dimensional array datum 
$T=(t_{i_1\cdots i_d})_{1\leq i_j\leq N_j, 1\leq j\leq d}$ is called a $d$-way tensor or
simply a $d$-tensor of format $N_1\times \cdots\times N_d$.
For a set $S$, the set of $N_1\times \cdots\times N_d$ tensors with entries in $S$
is denoted by $S^{N_1\times \cdots\times N_d}$.

Let $K$ be a field and $V_j$ an $N_j$-dimensional vector space over $K$ with fixed basis
$v_{j1}$, \ldots, $v_{jN_j}$ for $1\leq j\leq d$.
Then there is a one to one correspondence between $K^{N_1\times\cdots\times N_d}$
and $V_1\otimes\cdots\otimes V_d$ by
$(m_{i_1\cdots i_d})\leftrightarrow \sum_{i_1\cdots i_d}m_{i_1\cdots i_d}v_{1i_1}
\otimes\cdots\otimes v_{di_d}$.
A non-zero tensor corresponding to an element of 
$V_1\otimes \cdots\otimes V_d$ of the form $\aaa_1\otimes\cdots\otimes \aaa_d$
is called a rank 1 tensor.
For a general tensor $T$ of format $N_1\times\cdots\times N_d$, the rank of
$T$, denoted by $\rank T$, is by definition the minimum integer $r$
such that $T$ can be expressed as a sum of $r$ rank 1 tensors,
where we set the empty sum to be zero.
Thus, the rank is a measure of the complexity of a tensor.
Further, for a 2-tensor, i.e., a matrix, the rank is identical with the one
defined in linear algebra.

However, for the case where $d\geq 3$, the behavior of rank is much more
complicated than the matrix case.
In the matrix case, the rank is the maximum size of non-zero minors.
Thus, if $K$ is an infinite field, the set of $m\times n$ matrices with rank $\min\{m, n\}$ form a
Zariski dense open subset of $K^{m\times n}$.
However, there are non-empty Euclidean open subsets of $\RRR^{2\times 2\times 2}$
such that one consists of rank 2 tensors and the other one consists of
rank 3 tensors.
In particular, it is not possible to characterize the rank of a tensor
by 
vanishing and/or non-vanishing of 
polynomials.

Let $m$, $n$, $p$ be positive integers.
If the set of rank $r$ tensors of format $p\times n\times m$ over $\RRR$ 
contains a non-empty Euclidean open subset of $\RRR^{p\times n\times m}$, 
we say that $r$ is a typical rank 
of $p\times n\times m$ tensors over $\RRR$.
The set of typical ranks of $p\times n\times m$ tensors over $\RRR$ is
denoted as $\trank_\RRR(p,n,m)$ or simply $\trank(p,n,m)$.

If the base field is $\CCC$, the set of tensors of format $p\times n\times m$
with rank at most $r$ contains a non-empty Zariski open set if and only if
its Zariski closure is $\CCC^{p\times n\times m}$
(cf., Chevalley's Theorem, see e.g., \cite[p.\ 39]{har92}).
Therefore, there exists exactly one ``typical rank of $p\times n\times m$
tensors over $\CCC$''.
This is called the generic rank of $p\times n\times m$ tensors over $\CCC$
and denoted as $\grank_\CCC(p,n,m)$ or simply $\grank(p,n,m)$.

It is fairly easy to show that $\grank(p,n,m)=\min\trank(p,n,m)$
(see e.g., \cite[Chapter 6]{ssm16}).
Further $r\geq \grank(p,n,m)$ if and only if the $r$-th higher secant variety
of the image of Segre embedding $\PPP_\CCC\CCC^p\times\PPP_\CCC\CCC^n\times\PPP_\CCC\CCC^m
\to\PPP_\CCC\CCC^{p\times n\times m}$ is the whole space $\PPP_\CCC\CCC^{p\times n\times m}$,
where $\PPP_K V$ denotes the projective space 
consisting of one dimensional subspaces of the $K$-vector space $V$.
Thus, by counting the dimensions, we see that
$\grank(p,n,m)\geq\lceil \frac{mnp}{m+n+p-2}\rceil$.

Suppose that $3\leq m\leq n\leq p$.
Then $p\geq\lceil\frac{mnp}{m+n+p-2}\rceil$ if and only if $p\geq(m-1)(n-1)+1$.
Catalisano, Geramita, and Gimigliano \cite{cgg02} (see also \cite{cgg08})
proved that if $(m-1)(n-1)+1\leq p\leq mn$, then $\grank(p,n,m)=p$.
Thus, $\min\trank(p,n,m)=p$ in these cases.
ten Berge \cite{tb00} called a $p\times n\times m$-tensor with $(m-1)n<p<mn$ a tall
array or a tall tensor and proved that $\trank(p,n,m)=\{p\}$ for these cases
(see \cite[Chapter 6]{ssm16} for another proof).
Here we define a $p\times n\times m$-tensor a semi-tall tensor if
$(m-1)(n-1)+1\leq p\leq (m-1)n$.
We \cite{ssm13,sms15,sms17} studied the plurality of typical ranks 
of semi-tall tensors and proved
that if $(m-1)(n-1)+2\leq p\leq (m-1)n$, then $\trank(p,n,m)=\{p,p+1\}$ if there
exists a \nonsing\  bilinear map $\RRR^m\times \RRR^n\to \RRR^{mn-p}$ and
$\trank(p,n,m)=\{p\}$ otherwise, where a bilinear map
$\varphi\colon V_1\times V_2\to W$ is \nonsing\  if $\varphi(x,y)=0$ 
implies $x=0$ or $y=0$.

We also showed in \cite{sms17} that the former part of the above mentioned result
also valid in the case
where $p=(m-1)(n-1)+1$.
Therefore, the latter part of the above mentioned result in
the case where $p=(m-1)(n-1)+1$ is left open.
In this paper, we treat the case where $p=(m-1)(n-1)+1$ and show that 
$\trank(p,n,m)=\{p,p+1\}$ in several cases.
In particular, we show that the latter part does not valid in the case
where $p=(m-1)(n-1)+1$.


\section{Preliminaries}

Let $K$ be a field and $T=(t_{ijk})\in K^{\ell\times m\times n}$.
For $1\leq k\leq n$, we set $T_k=(t_{ijk})\in K^{\ell\times m}$
and denote $T=(T_1;\cdots;T_n)$.
For $P\in \glin(\ell,K)$ and $Q\in \glin(m,K)$, we set
$PTQ=(PT_1Q;\cdots;PT_nQ)$.
Note $\rank PTQ=\rank T$
by the definition of rank.

We first state the definition of the typical rank over $\RRR$.

\begin{definition}\rm
If the set of rank $r$ tensors over $\RRR$ of format $\ell\times m\times n$
contains a non-empty Euclidean open subset of $\RRR^{\ell\times m\times n}$,
then we say $r$ is a typical rank of $\ell\times m\times n$ tensors over $\RRR$.
We denote the set of typical ranks of $\ell\times m\times n$ tensors over
$\RRR$ by $\trank_\RRR(\ell,m,n)$ or simply $\trank(\ell,m,n)$.
\end{definition}

By the definition of the rank, we see the following fact.

\begin{lemma}
Let $n_1$, $n_2$ and $n_3$ be positive integers.
Then
$\trank(n_{i_1},n_{i_2},n_{i_3})=\trank(n_1,n_2,n_3)$
for any permutation $i_1$, $i_2$, $i_3$ of  1, 2, 3.
\end{lemma}

\begin{definition}\rm
For $T=(T_1;\cdots;T_n)\in K^{\ell\times m\times n}$, we set
$$
\fl_1(T)\define(T_1,\ldots, T_n)\in K^{\ell\times mn}
$$
and
$$
\fl_2(T)\define
\begin{pmatrix}T_1\\\vdots\\T_n\end{pmatrix}\in K^{\ell n\times m}
$$
and call flattenings of $T$.
\end{definition}
By the correspondence $K^{\ell\times m\times n}\leftrightarrow
V_1\otimes V_2\otimes V_3$, where $V_1$ (resp.\ $V_2$, $V_3$)
is a vector space over $K$ of dimension $\ell$ (resp.\ $m$, $n$)
with fixed basis,
flattenings correspond to natural isomorphisms
$V_1\otimes V_2\otimes V_3\to V_1\otimes (V_2\otimes V_3)$
and 
$V_1\otimes V_2\otimes V_3\to(V_1\otimes V_3)\otimes V_2$.
In particular,
$\rank (\fl_1(T))\leq \rank T$ and $\rank(\fl_2(T))\leq\rank T$.

\begin{definition}\rm
For $M=(\mmm_1,\ldots,\mmm_n)=
\begin{pmatrix}\mmm^{(1)}\\\vdots\\\mmm^{(\ell)}\end{pmatrix}
\in K^{\ell\times n}$,
we set
$M_{\leq j}\define(\mmm_1,\ldots, \mmm_{j})$,
${}_{j<}M\define(\mmm_{j+1},\ldots, \mmm_{n})$,
$M^{\leq i}\define\begin{pmatrix}\mmm^{(1)}\\\vdots\\\mmm^{(i)}\end{pmatrix}$
and
${}^{i<}M\define\begin{pmatrix}\mmm^{(i+1)}\\\vdots\\\mmm^{(\ell)}\end{pmatrix}$.
\end{definition}

\begin{definition}
Let $R$ be a commutative ring and $M\in R^{\ell\times n}$.
We denote by $I_t(M)$ the ideal of $R$ generated by $t$-minors of $R$.
\end{definition}


\section{A condition of an $n\times p\times m$-tensor to be of rank $p$}

From now on, let $m$, $n$ and $p$ be integers with $3\leq m\leq n$ and
$(m-1)(n-1)+1\leq p\leq mn$.
We set $u=mn-p$.

\begin{fact}
\begin{enumerate}
\item
$\grank(p,n,m)=p$.
In particular, $\min\trank(p,n,m)=p$ \cite{cgg02}.
\item
If $p>(m-1)n$, then $\trank(p,n,m)=\{p\}$ \cite{tb00}.
\item
Suppose $p\leq (m-1)n$.
If there exists a \nonsing\  bilinear map
$\RRR^m\times\RRR^n\to\RRR^u$, then $\trank(p,n,m)=\{p,p+1\}$.
Moreover, if $(m-1)(n-1)+2\leq p$, then the converse also hold true
\cite{ssm13,sms15, sms17}.
\end{enumerate}
\end{fact}

Therefore, the case where $p=(m-1)(n-1)+1$ and there is no \nonsing\ 
bilinear map $\RRR^m\times\RRR^n\to\RRR^u$ is still left open.
In the following, we consider the case where $p=(m-1)(n-1)+1$ and study 
if there are plural typical ranks of $p\times n\times m$ tensors over $\RRR$.

Before concentrating on the case where $p=(m-1)(n-1)+1$, we state
notations and a criterion of an $n\times p\times m$ tensor to be of rank $p$
in the case where $(n-1)(m-1)+1\leq p\leq (m-1)n$.
Note that $\trank(p,n,m)=\trank(n,p,m)$.

\begin{definition}\rm
\label{def:sigma nu}
We set
$\msVVV\define\{T\in\RRR^{n\times p\times m}\mid
\fl_2(T)^{\leq p}$ is \nonsing.$\}$,
$\msOOO\define\{Y\in\RRR^{u\times n\times m}\mid
{}_{p<}\fl_1(Y)$ is \nonsing.$\}$,
$\sigma\colon\msVVV\to\RRR^{u\times p}$,
$\sigma(T)=({}^{p<}\fl_2(T))(\fl_2(T)^{\leq p})^{-1}$,
$\nu\colon\msOOO\to\RRR^{u\times p}$,
$\nu(Y)=-({}_{p<}\fl_1(Y))^{-1}(\fl_1(Y)_{\leq p})$,
$\tau\colon\RRR^{u\times p}\to\msVVV$,
$\tau(W)=\fl_2^{-1}\begin{pmatrix}E_p\\ W\end{pmatrix}$
and
$\mu\colon\RRR^{u\times p}\to\msOOO$,
$\mu(W)=\fl_1^{-1}(W,-E_u)$.
\end{definition}
\begin{remark}\rm
$$
\begin{array}{lcr}
\RRR^{n\times p\times m}\supset\msVVV\\
&\!\!\!\!\!\!\!\!\!\!\tau\nwarrow\!\!\searrow\sigma\!\!\!\!\!\!\!\!\\
&&\RRR^{u\times p}\\
&\!\!\!\!\!\!\!\!\!\!\mu\swarrow\!\!\nearrow\nu\!\!\!\!\!\!\!\!\\
\RRR^{u\times n\times m}\supset\msOOO
\end{array}
$$
and
$\sigma(\tau(W))=\nu(\mu(W))=W$
for $W\in \RRR^{u\times p}$.
\end{remark}

\begin{definition}\rm
Let $\xxx=(x_1,\ldots,x_m)$ be a row vector of indeterminates,
i.e., $x_1$, \ldots, $x_m$ are independent indeterminates.
For $A=(A_1;\cdots;A_m)\in\RRR^{u\times n\times m}$,
we set $M(\xxx,A)\define x_1A_1+\cdots+x_mA_m\in \RRR[x_1, \ldots, x_m]^{u\times n}$.
\end{definition}

\begin{definition}\rm
Let $\aaa$ and $\bbb$ be column vectors with entries in $\RRR$ of 
dimension $m$ and $n$ respectively.
We set $\psi(\aaa,\bbb)\define(\aaa\kprod\bbb)^{\leq p}$,
where $\kprod$ denotes the Kronecker product, i.e.,
if $\aaa=\begin{pmatrix}a_1\\\vdots\\ a_m\end{pmatrix}$,
then
$\psi(\aaa,\bbb)=
\begin{pmatrix}a_1\bbb\\\vdots\\ a_{m-2}\bbb\\ a_{m-1}\bbb^{\leq p-(m-2)n}\end{pmatrix}$.

For $Y\in\RRR^{u\times n\times m}$, we define $U(Y)$ to be the vector
subspace of $\RRR^p$ generated by
$\{\psi(\aaa,\bbb)\mid M(\aaa\transpose,Y)\bbb=\zerovec\}$.
\end{definition}

\begin{lemma}
\label{lem:rank p}
For $T\in\msVVV$, 
$\rank T=p$ if and only if $\dim U(\mu(\sigma(T)))=p$.
\end{lemma}
\begin{proof}
Set $\mu(\sigma(T))=W=(W_1;\cdots;W_m)$
and $\ell=u-n$.
Then by \cite[Theorem 6.5 (1)$\iff$(3)]{sms17},
we see that $\rank T=p$ if and only if there are
$B=(\bbb_1, \ldots, \bbb_p)\in\RRR^{n\times p}$
and $p\times p$ diagonal matrices $D_1$, \ldots, $D_m$
such that
$$
(\ast)
\left\{
\begin{array}{lll}
D_k=\diag(d_{1k},\ldots, d_{pk})&\mbox{for $1\leq k\leq m$,}\\
(d_{j1}W_1+\cdots+d_{jm}W_m)\bbb_j=\zerovec&\mbox{for $1\leq j\leq p$ and}\\
\mbox{$\begin{pmatrix}
BD_1\\\vdots\\BD_{m-2}\\B^{\leq n-\ell}D_{m-1}\end{pmatrix}$ is \nonsing.}
\end{array}
\right.
$$

First suppose that there are $B=(\bbb_1, \ldots, \bbb_p)\in\RRR^{n\times p}$
and $D_1$, \ldots, $D_m$ which satisfy $(\ast)$.
If we set $\aaa_j=(d_{j1},\ldots, d_{jm})\transpose$ for $1\leq j\leq p$, then
$$
M(\aaa_j\transpose,W)\bbb_j=\zerovec
$$
for $1\leq j\leq p$ and
$$
(\psi(\aaa_1,\bbb_1), \ldots,\psi(\aaa_p,\bbb_p))=
\begin{pmatrix}
BD_1\\\vdots\\BD_{m-2}\\B^{\leq n-\ell}D_{m-1}\end{pmatrix}
$$
is \nonsing.
Therefore $\dim U(\mu(\sigma(T)))=p$.

Conversely, assume that $\dim U(\mu(\sigma(T)))=p$.
Then there are $\aaa_1$, \ldots, $\aaa_p\in\RRR^m$ and
$\bbb_1$, \ldots, $\bbb_p\in\RRR^n$ such that
$\psi(\aaa_1,\bbb_1)$, \ldots, $\psi(\aaa_p,\bbb_p)$ are linearly independent.
Set $\aaa_j=(d_{j1},\ldots,d_{jm})\transpose$ for $1\leq j\leq p$,
$D_k=\diag(d_{1k},\ldots,d_{pk})$ for $1\leq k\leq m$ and 
$B=(\bbb_1,\ldots,\bbb_p)\in\RRR^{n\times p}$.
Then it is easily verified that $B$ and 
$D_1$, \ldots, $D_m$ satisfy $(\ast)$.
\end{proof}


\section{Determinantal varieties and Bezout's theorem}

From now on, we consider the case where $p=(m-1)(n-1)+1$.
Then $u=m+n-2$.

\begin{definition}\rm
We set
$$
A_k\define\begin{pmatrix}O_{(k-1)\times n}\\ E_n\\ O_{(m-k-1)\times n}\end{pmatrix}
\in\RRR^{u\times n}
$$
for $1\leq k\leq m-1$ and
$$
A_{m}\define\begin{pmatrix}O_{(m-1)\times(n-1)}&-\eee_1\\ E_{n-1}&\zerovec\end{pmatrix}
\in\RRR^{u\times n},
$$
where $\eee_1=\begin{pmatrix}1\\0\\\vdots\\0\end{pmatrix}\in\RRR^{m-1}$
and $A=(A_1;\cdots;A_m)\in\RRR^{u\times n\times m}$.
\end{definition}

The next fact is the key lemma of this paper.

\begin{lemma}
\label{lem:key}
Let $y$ be an indeterminate and $a_1,\ldots, a_{m-1}\in\CCC$.
Then the following conditions are equivalent,
where $V_a(I)$ denotes the affine variety defined by an ideal $I$.
\begin{enumerate}
\item
$(a_1,\ldots, a_{m-1},-1)\in V_a(I_n(M(\xxx,A)))$.
\item
$y^{m-1}-a_{m-1}y^{m-2}-\cdots-a_2y-a_1$ is a factor of $y^u+1$.
\end{enumerate}
\end{lemma}

In order to prove this lemma, we need some preparation.
First we make the following

\begin{definition}\rm
Let $y$ be an indeterminate and 
$\{\mu_t\}_{t\geq 1}$ an infinite sequence of complex numbers.
We set $\msIII(\{\mu_t\}_{t\geq 1})
=\{f(y)\in\CCC[y]\mid f(y)=\sum_k c_k y^k$,
$\sum_k c_k\mu_{k+t}=0$ for any $t\geq 1\}$.
\end{definition}
It is easily verified that $\msIII(\{\mu_t\}_{t\geq 1})$ is an ideal of 
$\CCC[y]$.

Now let $a_1$, \ldots, $a_{m-1}\in\CCC$.
Set $\lambda_t=0$ for $1\leq t\leq m-2$, $\lambda_{m-1}=1$ and
$$
\lambda_{m-1+s}=
\det\begin{pmatrix}
a_{m-1}&a_{m-2}&\cdots&a_1\\
-1&\ddots&\ddots&&\ddots\\
&\ddots&\ddots&\ddots&&a_1\\
&&\ddots&\ddots&\ddots&\vdots\\
&&&\ddots&\ddots&a_{m-2}\\
&&&&-1&a_{m-1}
\end{pmatrix}
$$
for $s\geq 1$,
where the right hand side is an $s\times s$-determinant
(some $a_i$'s may not appear for small $s$).

By the first row expansion, we see the following

\begin{lemma}
\label{lem:rec rel}
For $t\geq m$, we have
$\lambda_t=\sum_{k=1}^{m-1}a_{m-k}\lambda_{t-k}$.
\end{lemma}
Set
$h(y)=y^{m-1}-a_{m-1}y^{m-2}-\cdots-a_2y-a_1$.
By the above lemma, we see that $h(y)\in\msIII(\{\lambda_t\}_{t\geq 1})$.
Further, since $\lambda_t=0$ for $1\leq t\leq m-2$ and $\lambda_{m-1}=1$,
there is no polynomial in $\msIII(\{\lambda_t\}_{t\geq 1})$ whose degree
is less than $m-1$ except the zero polynomial,
i.e.,
$\msIII(\{\lambda_t\}_{t\geq 1})$ is generated by $h(y)$.

Set 
$$
N=M((a_1,\ldots, a_{m-1},-1),A)
=\begin{pmatrix}a_1&&&&1\\
a_2&\ddots\\
\vdots&\ddots&\ddots\\
a_{m-1}&&\ddots&\ddots\\
-1&\ddots&&\ddots&a_1\\
&\ddots&\ddots&&a_2\\
&&\ddots&\ddots&\vdots\\
&&&-1&a_{m-1}
\end{pmatrix}
$$
and for integers $c_1$, \ldots, $c_n$ with $1\leq c_1<\cdots<c_n\leq u$,
we denote by $[c_1, \ldots, c_n]_N$ the maximal minor of $N$ 
consisting of the $c_1$-th, \ldots, $c_n$-th rows of $N$.

Now we state the following

\begin{lemma}
\label{lem:non full rank}
Under the notation above, the following  conditions are equivalent.
\begin{enumerate}
\item
\label{item:rank less than n}
$\rank N<n$.
\item
\label{item:minors vanish}
$[i,m,m+1,\ldots,u]_N=0$ for $1\leq i\leq m-1$.
\item
\label{item:lambda u+1}
$\lambda_{u+t}=0$ for $1\leq t\leq m-2$ and $\lambda_{u+m-1}=-1$.
\item
\label{item:rec u}
$\lambda_{u+t}=-\lambda_t$ for $t\geq 1$.
\item
\label{item:yu}
$y^u+1\in\msIII(\{\lambda_t\}_{t\geq 1})$.
\end{enumerate}
\end{lemma}
\begin{proof}
Let
$$
U=\begin{pmatrix}
a_{m-1}&a_{m-2}&\cdots&a_1&&1\\
-1&\ddots&\ddots&&\ddots\\
&\ddots&\ddots&\ddots&&a_1\\
&&\ddots&\ddots&\ddots&\vdots\\
&&&\ddots&\ddots&a_{m-2}\\
&&&&-1&a_{m-1}
\end{pmatrix}
$$
be a $u\times u$ matrix.
Then $\det({}^{t<}_{t<}U)=\lambda_{u+m-1-t}+\delta_{0,t}$ for $0\leq t\leq u-1$,
where $\delta_{0,t}$ is the Kronecker's delta.

\ref{item:rank less than n}$\Longrightarrow$\ref{item:lambda u+1}:
Since ${}_{m-2<}U=N$ and $\rank N<n$ by assumption, we see that
$\det({}^{t<}_{t<}U)=0$ for $0\leq t\leq m-2$.
Thus, we see that
$\lambda_{u+m-1-t}+\delta_{0,t}=0$ for $0\leq t\leq m-2$.

\ref{item:lambda u+1}$\Longrightarrow$\ref{item:minors vanish}:
We see by the first row expansions of $\det({}_{t<}^{t<}U)$ 
and $[t+1,m,\ldots, u]$
and the assumption that
\begin{eqnarray*}
0&=&\lambda_{u+m-1-t}+\delta_{0,t}\\
&=&\sum_{k=1}^{m-1}a_{m-k}\lambda_{u+m-1-t-k}+\delta_{0,t}\\
&=&\sum_{k=m-t-1}^{m-1}a_{m-k}\lambda_{u+m-1-t-k}+\delta_{0,t}\\
&=&\sum_{s=1}^{t+1}a_{t+2-s}\lambda_{u+1-s}+\delta_{0,t}\\
&=&[t+1,m,\ldots, u]_N
\end{eqnarray*}
for $0\leq t\leq m-2$.

\ref{item:minors vanish}$\Longrightarrow$\ref{item:rank less than n}
follows from the fact that
the last $n-1$ rows of $N$ are linearly independent,
\ref{item:lambda u+1}$\iff$\ref{item:rec u} follows from the facts that
$\lambda_1=\cdots=\lambda_{m-2}=0$ and $\lambda_{m-1}=1$ and
Lemma \ref{lem:rec rel}
and
\ref{item:rec u}$\iff$\ref{item:yu} follows from the definition of
$\msIII(\{\lambda_t\}_{t\geq 1})$.
\end{proof}

Since $(a_1,\ldots, a_{m-1},-1)\in V_a(I_n(M(\xxx,A)))$
if and only if $\rank N<n$ and $y^u+1\in\msIII(\{\lambda_t\}_{t\geq 1})$
if and only if $h(y)$ divides $y^u+1$, we see Lemma \ref{lem:key} by
Lemma \ref{lem:non full rank}.

Now we recall the following facts about determinantal varieties
(see e.g. \cite[p.151 and p.243]{har92}).

\begin{fact}
\label{fac:deg}
Let $X$ be a $u\times n$ matrix of indeterminates.
Then the projective variety in $\PPP_\CCC\CCC^{u\times n}$
defined by $I_n(X)$ has degree ${u\choose n-1}$ and codimension $u-n+1$.
\end{fact}
Note that ${u\choose n-1}={u\choose m-1}$ and $u-n+1=m-1$ since $u=m+n-2$.
Note also that there are ${u\choose m-1}$ monic factors of $y^u+1$ of degree $m-1$
in $\CCC[y]$.

In view of this fact, we make the following

\begin{definition}\rm
For $B=(B_1;\cdots; B_m)\in\CCC^{u\times n\times m}$,
we set
$\varphi_B\colon\CCC^m\to\CCC^{u\times n}$,
$(\alpha_1,\ldots, \alpha_m)\mapsto\alpha_1B_1+\cdots+\alpha_mB_m$.
\end{definition}
Then by Lemma \ref{lem:key}, Fact \ref{fac:deg} and Bezout's theorem,
we see the following

\begin{cor}
Let $\PPP_\CCC(\image\varphi_A)$ be the linear subspace of $\PPP_\CCC\CCC^{u\times n}$
defined by $\image\varphi_A$.
Then $\PPP_\CCC(\image\varphi_A)$ and $V_p(I_n(X))$ intersect transversely
at ${u\choose m-1}$ distinct points,
where $V_p(I)$ denotes the projective variety defined by the homogeneous ideal $I$.
\end{cor}
By the implicit function theorem, we see the following fact.

\begin{cor}
There is a Euclidean open neighborhood $\msUUU$ of $A$ in $\RRR^{u\times n\times m}$
such that if $B\in\msUUU$, then $\varphi_B$ is injective and the number of real
points of $V_p(I_n(M(\xxx,B)))\subset\PPP_\CCC\CCC^m$ is the number of real monic
polynomials of degree $m-1$ which divide $y^u+1$,
where we say a point of a complex projective space is real if all possible ratios
of its homogeneous coordintes
are real numbers.
\end{cor}
We denote the number of real monic polynomials of degree $m-1$ which divides 
$y^u+1$ by $\alpha(m,n)$.
Then we see the following

\begin{lemma}
\label{lem:alpha m n}
$$
\alpha(m,n)=
\left\{
\begin{array}{ll}
{u/2\choose (m-1)/2}&\quad\mbox{if $m$ and $n$ are odd,}\\
{(u-1)/2\choose (m-2)/2}&\quad\mbox{if $m$ is even and $n$ is odd,}\\
{(u-1)/2\choose (m-1)/2}&\quad\mbox{if $m$ is odd and $n$ is even and}\\
0&\quad\mbox{if $m$ and $n$ are even.}
\end{array}
\right.
$$
\end{lemma}

By replacing $\msUUU$ to a smaller neighborhood if necessary,
we may assume that $\rank (M(\aaa,B)_{\leq n-1})=n-1$
for any $B\in\msUUU$ and $\aaa\in\RRR^m\setminus\{\zerovec\}$
(cf.\ \cite[Corollary 4.20]{sms17}).
Then we have the following fact.

\begin{lemma}
\label{lem:number gen}
Suppose $B\in\msUUU$.
Then
$\#\{[\psi(\aaa,\bbb)]\in\PPP_\RRR\RRR^p\mid
M(\aaa,B)\bbb=\zerovec,
\aaa\in\RRR^m\setminus\{\zerovec\},
\bbb\in\RRR^n\setminus\{\zerovec\}\}
=\alpha(m,n)$,
where $[\xxx]$ denotes the point of $\PPP_\RRR\RRR^p$ defined by
$\xxx\in\RRR^p\setminus\{\zerovec\}$.
\end{lemma}


\section{Plural typical ranks of some formats of 3-tensors}

In this section, we show that in  certain formats of 3-tensors,
there are plural typical ranks.
We use the notation of the previous section.

First we recall the following fact.

\begin{fact}[{\cite[Proposition 2.4, Lemma 3.5, Theorems 7.3 and 8.1]{sms17}}]
\label{fact:bit dis}
If $m-1$ and $n-1$ are not bit-disjoint, then $\trank(n,p,m)=\{p,p+1\}$,
where two positive integers are bit-disjoint if there are no 1's
in the same place of their binary notation.
\end{fact}

\begin{example}\rm
\label{ex:known}
$\trank(n,p,m)=\{p,p+1\}$ in the following cases.
\begin{enumerate}
\item
Both $m$ and $n$ are even.
\item
$m=5$ and $n\equiv 5,6,7,8\pmod 8$.
\item
$m=6$ and $n\equiv 2,4,5,6,7,8\pmod 8$.
\item
$m=7$ and $n\equiv 3,4,5,6,7,8\pmod 8$.
\item
$m=8$ and $n\equiv 2,3,4,5,6,7,8\pmod 8$.
\item
$m=9$ and $n\equiv 9,10,11,12,13,14,15,16\pmod {16}$.
\end{enumerate}
\end{example}

Set
$
A''=(A_2;\cdots;A_{m-2};A_m;-A_{m-1};-A_{1})$.
Then there is a permutation matrix $P\in\glin(u,\RRR)$ 
such that
${}_{p<}\fl_1(PA'')=-E_u$.
Set $A'=PA''$ and $W_0=\fl_1(A')_{\leq p}$.
Further set $\rho\colon\RRR^{u\times n\times m}\to\RRR^{u\times n\times m}$,
$B=(B_1;\cdots;B_m)\mapsto P(B_2;\cdots;B_{m-2};B_m;-B_{m-1};-B_{1})$
and $\UUUUU=\mu^{-1}(\rho(\msUUU))$,
where $\mu$ is the map defined in Definition \ref{def:sigma nu}.
Then $\UUUUU$ is an open neighborhood of $W_0$.
Further, we see the following fact by Lemma \ref{lem:number gen}.

\begin{lemma}
If $T\in\sigma^{-1}(\UUUUU)$, then
$\#\{[\psi(\aaa,\bbb)]\in\PPP_\RRR\RRR^p\mid
M(\aaa,\mu(\sigma(T)))\bbb=\zerovec,
\aaa\in\RRR^m\setminus\{\zerovec\},
\bbb\in\RRR^n\setminus\{\zerovec\}\}
=\alpha(m,n)$.
\end{lemma}
Since $\sigma^{-1}(\UUUUU)$ 
contains $\tau(W_0)$, $\sigma^{-1}(\UUUUU)$ is not an empty set.
Therefore, we see by Lemma \ref{lem:rank p} that if $\alpha(m,n)<p$, then
there exists a non-empty Euclidean open subset $\sigma^{-1}(\UUUUU)$ of $\RRR^{n\times p\times m}$
consisting of tensors of rank larger than $p$.
Further, since we see by \cite[Theorem 8.1]{sms17}
that typical ranks of $n\times p\times m$ tensors are less than or equal to
$p+1$, we see the following fact.

\begin{lemma}
\label{lem:less than p}
If $\alpha(m,n)<p$, then $\trank(n,p,m)=\{p,p+1\}$.
\end{lemma}
%
%
Now we state the following

\begin{thm}
\label{thm:main}
Suppose $3\leq m\leq n$ and $p=(m-1)(n-1)+1$.
Then $\trank(n,p,m)=\{p,p+1\}$ in the following cases.
\begin{enumerate}
\item
$m=3$ or $m=4$.
\item
$m=5$ and $n\leq 26$ or $n=28$.
\item
$m=6$ and $n\leq 34$.
\item
$m=7$ and $n\leq 16$.
\item
$m=8$ and $n\leq 16$.
\end{enumerate}
\end{thm}
\begin{proof}
By Lemma \ref{lem:alpha m n} and computation, we see that $\alpha(m,n)<p$
in the following cases:
(1) $m=3$ or $m=4$,
(2) $m=5$ and $n\leq 26$ or $n=28$,
(3) $m=6$ and $n\leq 34$,
(4) $m=7$ and $n\leq 12$,
(5) $m=8$ and $n\leq 14$ and
(6) $m=9$ and $n=10$.
Thus, we see the result by Example \ref{ex:known} and Lemma \ref{lem:less than p}.
\end{proof}

\begin{remark}\rm
Set $m=3$ and $n=5$.
Then $m-1$ and $n-1$ are bit-disjoint.
However, by Theorem \ref{thm:main}, we see that 
$\trank(n,p,m)=\{p,p+1\}$.
Thus the converse of Fact \ref{fact:bit dis} does not valid.
\end{remark}

\end{document}